\documentclass[12pt]{amsart}
\usepackage{amssymb,epsfig,amsmath,latexsym,amsthm}
\usepackage{graphicx}
\usepackage{enumitem}
\usepackage{comment}
\theoremstyle{plain}
\newtheorem{theorem}{Theorem}[section]
\newtheorem{lemma}[theorem]{Lemma}
\newtheorem*{claim}{Claim}
\newtheorem{proposition}[theorem]{Proposition}
\newtheorem{corollary}[theorem]{Corollary}

\theoremstyle{definition}

\newtheorem{remark}[theorem]{Remark}

\newtheorem{example}[theorem]{Example}
\newtheorem*{thm1.1}{Theorem 1.1}

\usepackage{epsfig,color}
\makeatother

\theoremstyle{definition}

\def\fnum{equation}

\newtheorem{Def}[\fnum]{Definition}

\numberwithin{equation}{section}

\usepackage[hidelinks]{hyperref}

\begin{document}
\title[ The Isoperimetric Problem For  Lens Spaces ]{The Isoperimetric Problem For   Lens Spaces }

\author{Celso  Viana}\address{University College London  UCL Union Building,  25 Gordon Street, London WC1E 6BT}
\thanks{C.V. was  supported by The EPSRC Centre for Doctoral Training in Geometry and Number Theory:
The London School of Geometry and Number Theory}
\email{c.viana@imperial.ac.uk}

\begin{abstract}
We solve the isoperimetric problem in  the Lens spaces   with large fundamental group. Namely, we prove that the isoperimetric surfaces are geodesic spheres or tori of revolution about geodesics.
We also show that the isoperimetric problem in $L(3,1)$ and $L(3,2)$ follows from  the proof of the  Willmore conjecture by Marques and Neves.
\end{abstract}
\maketitle
\section{Introduction}

The isoperimetric problem is a classical  subject in Differential Geometry with its origin in ancient Greece. 
It consists in finding on a Riemannian manifold $M$ the regions that   minimize the perimeter among  sets  enclosing  the same volume. The solutions are called isoperimetric regions and their boundaries isoperimetric hypersurfaces. 
The Euclidean plane is historically the first space where the problem started to be investigated rigorously.
It is now a well known fact that  the round circles are the optimal curves for the problem. 
This geometric fact is often seen through    the following classical inequality:
\begin{equation}
L^2\geq 4\pi A,
\end{equation}
where $L$ and $A$ stand for the length  and the enclosed area  of a simple closed curve $\gamma: \mathbb{S}^1 \rightarrow \mathbb{R}^2$ respectively. 

The framework of geometric measure theory and its tools work successfully well in tackling the aspects of existence and regularity of this variational problem.
When $M^{n+1}$ is closed or homogeneous, then isoperimetric hypersurfaces do exist and are smooth up to a closed set of Hausdorff dimension $n-7$.
The regular part is a stable hypersurface of constant mean curvature.
This  major contribution to the isoperimetric problem  was achieved thanks to the efforts of many people, including F. Almgren,  R. Schoen, L. Simon, F. Morgan, and others  (see \cite{M} for a comprehensive list).
Despite the long history of the problem, it remains largely open with few   3-manifolds where the problem is completely understood.

The simply connected space forms, $\mathbb{S}^{n+1}$, $\mathbb{R}^{n+1}$ and $\mathbb{H}^{n+1}$, are the most appealing spaces to begin  the study of the isoperimetric  problem.
It turns out that their symmetries  are enough to characterize the geodesic spheres as the   isoperimetric hypersurfaces.

A complete solution on $\mathbb{S}^2\times\mathbb{S}^1$ with the standard product metric can be found in \cite{PR}. 
For other homogeneous manifolds with certain product structure, such as $\mathbb{H}^2\times\mathbb{R}$, $\mathbb{H}^2\times \mathbb{S}^1$, and $\mathbb{S}^n\times \mathbb{R}$, see \cite{HH}, \cite{PR}, and \cite{P} respectively. The case $\mathbb{S}^1\times \mathbb{R}^n$ is also treated in \cite{PR} and they show that when $n\geq 9$  unduloids are minimizers rather than cylinders for certain volumes. One key idea exploited   in the results listed above is the use of  symmetry to reduce the problem to an ODE analysis. 
The case $T^2 \times \mathbb{R}$ where $T^2$ is a flat torus is   not solved in full;  great progress can be found in \cite{RR1} and in \cite{HPRR}. More generally, it is known  that    boundaries of small isoperimetric regions in closed manifolds are  nearly round spheres, see \cite{MJ}.
To finish this brief and not  exhaustive  account of results on the isoperimetric problem we mention that    Bray and  Morgan ( \cite{B} and  \cite{BM}) classified the horizon  homologous   isoperimetric surfaces   in the Schwarzschild manifold. The work in \cite{B}   highlighted the interesting relationship between isoperimetric surfaces and the concept of mass in general relativity.  

We will be interested in spherical space forms in this paper.
A significant result in this direction was given in \cite{RR} where A. Ros and M. Ritor\'{e}   solved the isoperimetric problem in the projective space $\mathbb{RP}^3$. They show that the solutions are geodesic spheres or flat tori. 
Later, A. Ros \cite{R} used the above result to give a proof of the Willmore conjecture in $\mathbb{S}^3$ for  the special case of surfaces that are invariant by the antipodal map.

The real projective space is a special case of an important family of Riemannian manifolds, namely the Lens spaces $L(p,q)$.
These are elliptic space forms obtained as a quotient of $\mathbb{S}^3$ by a finite group of isometries  that are isomorphic to $\mathbb{Z}_p$ but which also depend on $q$.
They are, along with $\mathbb{S}^2\times \mathbb{S}^1$,  characterized by having Heegaard genus one. 

We give a complete solution for the isoperimetric problem in the Lens spaces with large fundamental group:

\begin{theorem}\label{theorem 1}
\textit{There exists a positive integer $p_0$ such that for every $p\geq p_0$ and every $q\geq 1$ the isoperimetric surfaces in  $L(p,q)$ are  either geodesic spheres or tori of revolution about geodesics.}
\end{theorem}
The isoperimetric problem in dimension three  was previously solved for only a finite number of  non-diffeomorphic 3-manifolds. 

We also add to the literature the observation that   the proof of the Willmore  conjecture  by Marques and Neves \cite{MN} can be applied to extend the work of Ros and Ritor\'{e} \cite{RR},  on the classification of stable CMC surfaces in $\mathbb{RP}^3$, to  $L(3,1)$ and $L(3,2)$: \footnote{ $L(3,1)$ and $L(3,2)$ are diffeomorphic but not isometric.}
\begin{theorem}\label{theorem 2}
\textit{The  immersed stable  CMC surfaces in $L(3,1)$  and $L(3,2)$ are  either geodesic spheres or  flat tori. Moreover,  the minimal Clifford torus is, up to ambient isometries,  the only index one minimal surface in $L(3,1)$ and $L(3,2)$.}
\end{theorem}

 The idea of the proof of Theorem \ref{theorem 1} is  as follows. 
 Stability implies that every isoperimetric surface is connected and its genus is  $0,1,2$ or $3$.
 If follows from a classical result of Hopf that if the genus is 0, then it is a geodesic sphere. From \cite{RR} we know that if the genus is 1, then it is  flat, and this forces the surface to be of revolution about a geodesic. We are left to rule out  other topological types. To do so we argue by contradiction. We assume that there exists a sequence of Lens spaces with increasing fundamental group containing  isoperimetric surfaces of  genus $2\, \text{or}\, 3$. After a suitable rescaling on the metrics we use compactness results to obtain a limit for the sequence of Lens spaces which will be a flat three manifold of rank one. In the same way, the sequence of isoperimetric surfaces will converge to a flat surface in the respective ambient manifold. The topology of the surfaces will force the limit to be an union of  planes. On the other hand, the minimization property rules this configuration out. 

The arguments  in the proof of Theorem \ref{theorem 1} generalize naturally for   the Berger Spheres $\mathbb{S}_{\varepsilon}^3$. This is a well known one parameter family of homogeneous metrics on the 3-sphere; the case $\varepsilon =1$ corresponds to the round metric. 
 All  the Hopf fibers in $\mathbb{S}_{\varepsilon}^3$ have the same length  $2\pi\,\varepsilon$. 
When $\varepsilon$ is not too small,   spheres are the only solutions of the isoperimetric problem, see \cite{TU}. When $\varepsilon$ is  very small,  it is pointed out in \cite{TU} that   some tori  are better candidates to solve the isoperimetric problem rather than spheres for certain  volumes.
\begin{theorem}
\textit{There exists $\varepsilon_0>0$ such that  for every $\varepsilon< \varepsilon_0$ the isoperimetric surfaces in the Berger spheres $\mathbb{S}_{\varepsilon}^3$ are either rotationally invariant spheres or tori. }
\end{theorem}
\begin{remark} 
 It follows from the work in \cite{MHH} and \cite{PR}  that the minimal Clifford torus in $L(p,1)$  (resp. $\mathbb{S}_{\varepsilon}^3$) is isoperimetric   for every $p\geq 3$  (resp. $\varepsilon\leq \frac{1}{3}$), see Section 3.
\end{remark}  
 
\noindent\textit{Acknowledgements.}
I am grateful to my  advisor  Andr\'{e} Neves for his constant encouragement and support during the course of this work. I also thank Jason Lotay  for many helpful comments on an earlier version of this manuscript.

\section{Preliminaries}
Let $(M^{n+1},g)$ be an orientable  Riemannian manifold of dimension $n+1$.  The $n+1$-dimensional Hausdorff measure of  a region $\Omega\subset M$  is denoted by $|\Omega|$. Similarly, we denote the $n$-dimensional Hausdorff measure of the hypersurface $\partial \Omega \subset M$ by $|\partial\Omega|$.
The class of  regions considered here are  those of finite perimeter, see \cite{S}.
\subsection{Isoperimetric}
A region $\Omega\subset M$ is called an \textit{Isoperimetric region} if
\[|\partial \Omega|=\inf \{|\partial \Omega^{\prime}|\,:\,\Omega^{\prime} \subset M \quad \textit{and} \quad |\Omega^{\prime}|=|\Omega|\}.\]
In this case, the hypersurface $\Sigma=\partial \Omega$ is called an \textit{Isoperimetric hypersurface}.

In general, the existence of isoperimetric surfaces is handled by a compactness theorem from geometric measure theory. For non-compact manifolds one needs to be careful since a minimizing sequence of regions of fixed volume may drift off to infinity.
We recommend \cite{M} for a recent reference on the regularity of isoperimetric hypersurfaces:

\begin{theorem}
\textit{Let $(M^{n+1},g)$ be a  closed Riemannian manifold. For every $0<t< \text{vol}(M)$ there exists an isoperimetric region $\Omega$ satisfying $|\Omega|=t$. Moreover, $\Sigma=\partial \Omega$ is smooth up to a closed set of Hausdorff dimension $n-7$.}
\end{theorem}
\subsection{Stability}
Isoperimetric hypersurfaces are stable critical points of the area functional for variations that preserve the enclosed volume; thus, the regular part of isoperimetric hypersurfaces has constant mean curvature.  
More generally, we say an isometric immersion $\phi: \Sigma^n \rightarrow M^{n+1}$ has \textit{constant mean curvature} (CMC) if it is a critical point of the area functional for volume preserving variations. Critical points are called \textit{stable CMC} if the second derivative of the area  is non-negative for such variations. 

Equivalently,  $\phi$ is stable if  for every $f \in C^{\infty}(\Sigma)$ with compact support such that $\int_{\Sigma} f\,dvol_{\Sigma}=0$,   we have
\begin{equation}\label{stability inequality}
I(f,f)=-\int_{\Sigma}fL\,f\,d_{\Sigma} :=\int_{\Sigma} |\nabla f|^2 - (Ric(N,N)+ |A|^2)\,f^2\, d_{\Sigma}\geq 0.
\end{equation}
$N$ is the unit normal vector of $\Sigma$ and $A$ is the \textit{second fundamental form} of the immersion $\phi$. The \textit{mean curvature} of  $\Sigma$, denoted by $H$, is defined by $2\,H=\textit{trace}(A)$.

The  study of immersed stable CMC hypersurfaces started in \cite{BC} and \cite{BCE} with a new characterization of the geodesic spheres in the  simply connected space forms $\mathbb{R}^{n+1}$, $\mathbb{S}^{n+1}$ and  $\mathbb{H}^{n+1}$. 
The classification  of stable CMC surfaces   is often a way to
approach the isoperimetric problem in  reasonable spaces. 
With this purpose in mind, A. Ros and M. Ritor\'{e} in \cite{RR} used the Hersch-Yau trick   to study orientable stable CMC surfaces on three manifolds of positive Ricci curvature.

\begin{theorem}[Ros-Ritor\'{e} \cite{RR}]\label{up bound w energy, g=2}
\textit{Let $(M,g)$ be a closed three manifold with positive Ricci curvature. If  $\phi: \Sigma \rightarrow M^3$  is a stable CMC immersion,   then $g(\Sigma)\leq 3$. Moreover, if $g(\Sigma)=2$ or $3$, then
\[\big(\frac{1}{2}\inf_{\Sigma}Ric_M+ H^2\big)|\Sigma|\leq 2\pi.\] }
\end{theorem}

\begin{proposition}\label{wilmore energy genus 2}
\textit{Let $\phi: \Sigma \rightarrow M^3$  be a stable immersion  with  constant mean curvature $H$ into an elliptic space form $M=\mathbb{S}^3/G$. Then
\begin{enumerate}
\item If $g(\Sigma)=2$ or $3$, then $\big(1+H^2\big)|\Sigma|\leq 2\pi$.
\item  If $g(\Sigma)=2$ or $3$ and $|G|\leq 4 $, then $\phi$ is an embedding. Moreover,  if   $|G|\leq 6$, then the pullback of $\Sigma$, through the covering map $\Pi: \mathbb{S}^3\rightarrow M^3$, is connected.
\item {If $|G|=2$ or $3$, then $g(\Sigma)=\,0$ or $1$.}
\end{enumerate}}
\end{proposition}
\begin{proof}
The first statement follows from the previous theorem since $Ric_M=2$. Let $\phi_{*}: \pi_1(\Sigma)\rightarrow \pi_1(M)$ be the induced map in fundamental groups. As  $K=Ker(\phi_{*})$ has finite index there exists a
finite covering $\psi:\widetilde{\Sigma}\rightarrow \Sigma$ such that $Im\big(\psi_{*}\big)=K$ and
$\big(\phi\circ\psi\big)_{*}=0$. This means there exists a lifting of this map into $\mathbb{S}^3$ and we denote it by
$\widetilde{\phi}:\widetilde{\Sigma}\rightarrow \mathbb{S}^3$. It follows that
$\big(1+H^2\big)|\widetilde{\Sigma}|\leq |G|\,2\pi$. If $\phi$ is not an embedding,   then $\widetilde{\phi}$ is  not embedding either. By the work of Li and Yau \cite{LY} the
Willmore energy of the  immersed surface $\widetilde{\Sigma}$, i.e. $\mathcal{W}(\widetilde{\Sigma})=\int_{\widetilde{\Sigma}}\big(1+H^2\big)dvol_{\Sigma}$,  is strictly greater\footnote{The case $\mathcal{W}(\Sigma)=8\pi$ is discussed in \cite{RR}.} than $8\pi$. Therefore, if $|G|\leq 4$, we obtain a contradiction and $\widetilde{\phi}$ is an embedding.  Moreover, for closed surfaces with genus  greater than or equal to $1$ in
$\mathbb{S}^3$ the Willmore conjecture, recently proved in \cite{MN}, states that $\mathcal{W}(\Sigma)\geq 2\pi^2$. Let $\cup_{i=1}^l\widetilde{\Sigma}_i$  be the pre-image of $\Sigma$  by the universal covering map, then
\[2\,l\,\pi^2\leq \sum_{i=1}^l\mathcal{W}\big(\widetilde{\Sigma}_i\big)= |G| \mathcal{W}(\Sigma)\leq |G|2\pi \Rightarrow \frac{|G|}{l}\geq \pi\]
Therefore, if $|G|\leq 6$,  then $l=1$ and $|G|\geq \pi$. In particular, if $|G|=2$ or $3$, then there exist no stable CMC surface $\Sigma$ with $g(\Sigma)\geq 2$. 
\end{proof}
\begin{Def}
For each $r\in(0,\frac{\pi}{2})$ we define the Clifford Torus $T_r$ as:
\begin{eqnarray}\label{clifford torus}
T_r= \mathbb{S}^1(\cos(r))\times \mathbb{S}^1(\sin(r))\subset \mathbb{S}^3.
\end{eqnarray}
\end{Def}
 
Every flat tori with constant mean curvature in $\mathbb{S}^3$ is congruent to a Clifford torus  $T_r$. This follows from the Rigidity theorem, pg. 49 in \cite{C}, for they have the same second fundamental form.

\begin{corollary}
\textit{The  stable CMC surfaces in $L(3,1)$ and $L(3,2)$ are totally umbilical spheres or flat tori. In addition, the index one minimal surfaces in $L(3,1)$ and $L(3,2)$ are congruent to the the projection of  minimal Clifford torus.}
\end{corollary}  
\begin{proof}
Let $\Sigma \subset L(3,q)$, $q=1,2$, be in the conditions of the corollary. 
By  Proposition \ref{wilmore energy genus 2}, $g(\Sigma)=0$ or $1$. If $g(\Sigma)=0$, then it follows from the Hopf holomorphic quadratic differential that $\Sigma$ is totally umbilical. If $g(\Sigma)=1$, then it is proved in \cite{RR} that $\Sigma$ is flat. 
\end{proof}
\subsection{Isoperimetric profile}
The isoperimetric properties of $M$ can be encapsulated in a single function called the  \textit{Isoperimetric profile}. This is the function $I_M: [0,\text{vol}(M)] \rightarrow [0,+ \infty)$ defined by 
\begin{eqnarray}
I_M(v)= \inf\{ |\partial \Omega|\,:\,\Omega \subset M \quad \text{and}\quad |\Omega|=v\}.
\end{eqnarray}
We finish the  section with some well known facts on the analytic nature of $I_M$. These will be used later in Section 3. 

Let $\Omega$ be an isoperimetric region in $M$ such that $|\Omega|=v$ for some $v \in (0,\text{Vol}(M))$. The function $I_M$ has left and right derivative $(I_M)_-^{\prime}(v)$ and $(I_M)_+^{\prime}(v)$. In addition, if $H$ is the mean curvature of $\Sigma=\partial \Omega$ in the direction of the inward unit vector, then 
\begin{eqnarray} \label{first derivative}
(I_M)_{+}^{\prime}(v)\leq2H\leq (I_M)_{-}^{\prime}(v).
\end{eqnarray}
The second derivative also exists but  weakly in the sense of comparison functions. More precisely, we say $f^{\prime\prime}\leq h$ weakly at $x_0$ if there exists a smooth function $g$ such that $f \leq g$, $f(x_0)=g(x_0)$, and $g^{\prime\prime}\leq h$. In this sense we have
\begin{eqnarray}\label{second derivative}
I_M(v)^2\, I_M^{\prime\prime}(v)+\int_{\Sigma}\big(Ric_g(N,N)+|A|^2\big)\,d_{\Sigma}\leq\,0.
\end{eqnarray}
The equations  (\ref{first derivative}) and (\ref{second derivative}) are first presented on  \cite{BP}.
We sketch the proof of  (\ref{first derivative}) and (\ref{second derivative}).

Let $\Sigma_V$ be the variation $\Sigma_t= exp_{\Sigma}(tN)$   of $\Sigma$   reparametrized  in terms of the enclosed volume $v(t)$. In addition, let $\phi(t)$ (resp. $\phi(v)$) be the area of $\Sigma_t$ (resp. $\Sigma_v$). 
By the first variation formula for the  area and volume we have  $\phi^{\prime}(0)= 2\,H\, |\Sigma|$  and 
$v^{\prime}(0)=|\Sigma|$ respectively.
Since $\phi^{\prime}(t)=\phi^{\prime}(v)v^{\prime}(t)$, we conclude that
$\phi^{\prime}(v(0))= 2H$ and also that
$v^{\prime}(0)^2\phi^{\prime\prime}(v(0))=\phi^{\prime\prime}(0)-\phi^{\prime}(v(0))v^{\prime\prime}(0)$.
On the other hand,  the second derivative of area for general variations implies the following:
\begin{eqnarray*}
\phi^{\prime\prime}(0)&=&
-\int_{\Sigma}1\,L\,1\,d_{\Sigma} +2H\,v^{\prime\prime}(0)\\
&=& - \int_{\Sigma}\big(Ric_g(N,N)+|A|^2\big)d_{\Sigma} +2H\, v^{\prime\prime}(0).
\end{eqnarray*}
Hence, in the sense of comparison functions,  (\ref{second derivative}) follows from:
\begin{eqnarray}\label{second derivative area}
\phi(v(0))^2\, \phi^{\prime\prime}(v(0))+\int_{\Sigma}\big(Ric_g(N,N)+|A|^2\big)d_{\Sigma}=\,0.
\end{eqnarray}

\section{The Isoperimetric problem in the Lens Spaces}

In order to define the Lens spaces, we first  recall the round three sphere as:
\[\mathbb{S}^3=\{(z,w)\in \mathbb{C}^2\,:\, |z|^2+|w|^2=1 \}.\]
Fix $p,q$   integers with the following property $1\leq q< p$ and $gcd(p,q)=1$. Let $\mathbb{Z}_p$ be the group $\mathbb{Z}/p\mathbb{Z}$ acting on $\mathbb{S}^3$ as follows:
\begin{equation}\label{Zp}
m\in \mathbb{Z}_p\mapsto m\cdot(z,w)=(e^{\frac{2\pi i qm}{p}}z, e^{\frac{2\pi i m}{p}}w).
\end{equation}
The group $\mathbb{Z}_p$ acts freely and properly discontinuous on $\mathbb{S}^3$. The orbit space $\mathbb{S}^3/\mathbb{Z}_p$ is a closed three manifold called the \textit{Lens space}, it is  denoted by $L(p,q)$.

The Hopf fibration, which is the Riemannian submersion $h: \mathbb{S}^3\rightarrow \mathbb{S}^2(\frac{1}{2})$ defined by $h(z,w)=\frac{z}{w}$, can be extended naturally to $L(p,q)$. 
Indeed, the group $\mathbb{Z}_p$ acts on the set of  \textit{Hopf fibers} through the following 
cyclic action on $\mathbb{S}^2(\frac{1}{2})$:
\[\Gamma_p= \langle e^{\frac{2\pi i (q-1)}{p}}\rangle : \lambda \in \mathbb{C}\cup \{\infty\} \mapsto  e^{\frac{2\pi i (q-1)}{p}} \cdot \lambda \in \mathbb{C}\cup \{\infty\}.  \]
The Hopf fibration for $L(p,q)$ is then defined as
$h: L(p,q)\rightarrow \mathbb{S}^2(1/2)/\Gamma_p$. The set $\mathbb{S}^2(1/2)/\Gamma_p$ is a two dimensional orbifold with  conical singularities at the north and south pole.  

The preimage  of horizontal slices of $\mathbb{S}^2(1/2)/\Gamma_p$ via $h$ correspond to the  Clifford torus described in (\ref{clifford torus}). They are natural candidates to solve the Isoperimetric problem in $L(p,q)$.

\subsection{Comments on Steiner Symmetrization for Lens Spaces}

Steiner  and Schwarz symmetrization  theorems were proved in \cite{MHH} for certain fiber bundles such as the Lens spaces.  
To explain this symmetrization procedure we restrict  to the case $L(p,1)$ where the Hopf fibration $h: L(p,1) \rightarrow \mathbb{S}^2(\frac{1}{2})$ is a smooth Riemannian submersion. 

The symmetrization consist of associating to each set of finite perimeter $R\subset  L(p,1)$  the set $Sym(R)$ in the product manifold $\mathbb{S}^2(\frac{1}{2})\times \mathbb{S}^1(\frac{1}{p})$ defined by replacing the slice of $R$ in each fiber with a ball of the same volume about the respective fiber in the product. The coarea formula for Riemannian submersions implies that $Sym(R)$ encloses  the same amount of volume as $R$. It is  proved in \cite{MHH} that $Sym(R)$ has no greater perimeter than $R$.

One immediate consequence is that $I_{L(p,1)}\geq I_{\mathbb{S}^2(\frac{1}{2})\times \mathbb{S}^1(\frac{1}{p})}$. Applying the classification of the isoperimetric problem on 
$\mathbb{S}^2(\frac{1}{2})\times \mathbb{S}^1(\frac{1}{p})$, \cite{PR}, we conclude 
that $I_{L(p,1)}= I_{\mathbb{S}^2(\frac{1}{2})\times \mathbb{S}^1(\frac{1}{p})}$ in a interval around $V=\frac{\text{Vol}(L(p,1))}{2}$.
In particular, the minimal Clifford torus is isoperimetric in $L(p,1)$ for every $p\geq 3$.
The isoperimetric profiles, however, do not coincide as  the profile of geodesic spheres on the respective spaces are different. Therefore, this technique is not enough to completely solve the isoperimetric problem.

It is  important to point out that, for general Lens spaces $L(p,q)$, there is no analogue of \cite{PR} for $\mathbb{S}^2(\frac{1}{2})/\Gamma_p\times \mathbb{S}^1(\frac{1}{n_p})$ which is a manifold having codimension two singularities.

\subsection{Some aspects of Lens Spaces} 

For every $x \in L(p,q)$ the injectivity radius of $L(p,q)$ at $x$ satisfies $\text{inj}_xL(p,q)\geq \frac{\pi}{p}$, with equality only at points in the critical fibers. Indeed, for $\theta=e^{\frac{2\pi i }{p}}$ we have:
\begin{eqnarray*}
a^2:=d_{\mathbb{R}^4}^2[(\theta^{kq} z, \theta^k w),(z,w)]&=&
|\theta^{kq}-1|^2|z|^2+|\theta^k-1|^2|w|^2\geq |\theta-1|^2\\
d_{\mathbb{S}^3}[(\theta^{kq} z, \theta^k w),(z,w)]&=&2\arcsin(\frac{a}{2})\geq 2\arcsin\big({\frac{2}{2}\sin{\frac{2\pi}{2p}}}\big)=\frac{2\pi}{p}.
\end{eqnarray*}
However, is not true in general that $inj(L(p,q), x)=O(\frac{1}{p})$ as $p \rightarrow \infty$.

\begin{example}
Let's consider  $L(k^2,k+1)$, $k\in \mathbb{Z}_{+}$. We show that the injectivity radius at  points  far away from the critical fibres are $O(\frac{1}{k})$. If  the round metric is scaled by the factor $k^2$, then we have the Riemannian submersion:
\[h: \big(L(k^2,k+1), k^2 g_0, x_k\big)\rightarrow \big(\mathbb{S}^2/\mathbb{Z}_k,k^2 g_{\mathbb{S}^2}, h(x_k)\big).\] 
The fibers have constant length $2\pi$ except the critical fibres which have length $\frac{2\pi}{p}$. The right hand side will converge, as $k \rightarrow \infty$, to $\mathbb{S}^1\times \mathbb{R}$.  It follows from  the coarea formula for Riemannian submersions  that the volume of the  geodesic ball $B_{4\pi}(x_k)$ in $\big(L(k^2,k+1), k^2 g_0, x_k\big)$ is bounded from below. Therefore, by Cheeger's inequality, Lemma 51 in \cite{PP}, the injectivity radius of the sequence $\big(L(k^2,k+1), k^2 g_0, x_k\big)$ is bounded from below. This sequence converges to a flat $T^2\times \mathbb{R}$.
\end{example}

If $x, y \in T_{\frac{\pi}{4}}/\mathbb{Z}_p \subset L(p,q)$, then $d_{L(p,q)}(x,y)\geq C d_{T_{\frac{\pi}{4}}/\mathbb{Z}_p}(x,y)$ for some  constant $C>0$ independent of $p,q$. Thus intrinsic and extrinsic distances on $T_{\frac{\pi}{4}}/\mathbb{Z}_p$ are equivalent.

\begin{lemma}\label{injective radius}
\textit{If $x \in T_{\frac{\pi}{4}}/\mathbb{Z}_p\subset L(p,q)$ and the extrinsic diameter of $T_{\frac{\pi}{4}}/\mathbb{Z}_p$ in $L(p,q)$ is bounded from below, then $inj_{x}L(p,q)=O(\frac{1}{p})$.}
\end{lemma}
\begin{proof}
Let $\lambda_p= \frac{1}{inj_x L(p,q)}$ and recall that $\frac{1}{\lambda_p}\geq \frac{\pi}{p}$. 
Without loss of generality, let's assume that $\text{diameter}_{L(p,q)}(T_{\frac{\pi}{4}}/\mathbb{Z}_p)\geq 1$.
Hence, under the rescaled metric $\lambda_p^2\, g_{\mathbb{S}^3}$, the extrinsic diameter of $T_{\frac{\pi}{4}}/\mathbb{Z}_p$ is greater than or equal to $\lambda_p$. 
Let $\gamma_p(t)$ be a geodesic segment realizing the intrinsic diameter of $T_{\frac{\pi}{4}}/\mathbb{Z}_p$. Thus, we can find disjoint balls $B_R(x_i)\subset L(p,q)$, with $R< \frac{C}{2}$, $x_i \in \gamma_p(t)$, and $i=1,\ldots, [\lambda_p]+1$. Hence, $$\sum_{i=1}^{[\lambda_p]+1}\mathcal{H}^2(B_R(x_{i_0})\cap T_{\frac{\pi}{4}}/\mathbb{Z}_p) \leq |T_{\frac{\pi}{4}}/\mathbb{Z}_p|= \lambda_p^2\,\frac{2\pi^2}{p}.$$
Therefore,  there exists $i_0\in \{1,\ldots, [\lambda_p]+1\}$ such that:
\[C_1\leq \mathcal{H}^2(B_R(x_{i_0})\cap T_{\frac{\pi}{4}}/\mathbb{Z}_p)\leq \frac{2\pi^2\cdot \lambda_p}{p}.\]
The first inequality follows from the Monotonicity Formula, Proposition \ref{monotonicity} in the Appendix, applied to $T_{\frac{\pi}{4}}\subset (L(p,q),\lambda_p^2\,g_{\mathbb{S}^3},x)$. This finishes the proof of the lemma.
\end{proof}

Let's use the notation $\mathbb{Z}_{p}^{q}$ to represent the group $\mathbb{Z}_p$ acting on $\mathbb{S}^3$ and its dependence on the parameter $q$.
By using the toroidal coordinate system for $\mathbb{S}^3$,
\[\varphi_r : \mathbb{R}^2 \rightarrow T_r: \varphi_r(u,v)=(\cos(r) e^{2\pi i u}, \sin(r)e^{2\pi i v})\in  \mathbb{S}^3,\]
the action of $\mathbb{Z}_p^q$ on $T_r$ corresponds to the following action  on $\mathbb{R}^2$:
\[(u,v) \longmapsto (u+ \frac{k q}{p}, v + \frac{k}{p}).\] 
In these coordinates , the $\mathbb{Z}_p^q$ orbit at the point $(z_0,w_0)=\varphi (\mathbb{Z}\times \mathbb{Z})\in T_{r}$  is given by:
\begin{eqnarray}\text{Orbit}_{p,q}(z_0,w_0)=\{(m,n) + k(\frac{1}{p}, \frac{q}{p}): m,n,k \in \mathbb{Z}\}.
\end{eqnarray}

\begin{lemma}\label{subsequences}
\textit{Given a sequence $\{L(p,q)\}_{p \in \mathbb{N}}$, there exist $b,m_0,n_0 \in \mathbb{Z}$ and  a subsequence $\{L(p_l,q_l)\}_{l \in \mathbb{N}}$ such that one of the following holds:
\begin{enumerate}
\item For every $(z_0,w_0) \in T_r$, $\varphi_r(\text{Orbit}_{p_l,q_l}(z_0,w_0))$ is becoming dense on $T_{r}$ as $l\rightarrow\infty$.
\item $\varphi_r(\text{Orbit}_{p_l,q_l}(z_0,w_0))$ is contained in $b$ integral curves  of $X(z,w)=(m_0 \sqrt{-1}\, z, n_0 \sqrt{-1}\,w)\in \mathcal{X}(\mathbb{S}^3)$.
\end{enumerate}}
\end{lemma} 
\begin{proof} 
To prove the lemma it is enough to consider $(z_0,w_0)\in T_{\frac{\pi}{4}}$. If there is a subsequence for which the diameter of $T_{\frac{\pi}{4}}/ \mathbb{Z}_p^q$ is going to zero as $p \rightarrow \infty$, then 
$\varphi_{\frac{\pi}{4}}(\text{Orbit}_{p,q}(z_0,w_0))$ is clearly  becoming dense on $T_{\frac{\pi}{4}}$ and item 1 is proved. 

Let's consider  now the case where the diameter of $T_{\frac{\pi}{4}}$ in $L(p,q)$ is bounded away from zero. From the equivalence between extrinsic and intrinsic distance and by Lemma \ref{injective radius} we conclude that the Euclidean injectivity radius satisfies $inj_{(z_0,w_0)} T_{\frac{\pi}{4}}/\mathbb{Z}_p= O(\frac{1}{p})$. 
In particular, there exist $k_p,m_p,n_p \in \mathbb{Z}$ such that 
\[0< ||k_p(\frac{1}{p}, \frac{q}{p})-(m_p,n_p)||_{\mathbb{R}^2}\leq \frac{C}{p} .\]
Therefore, there exists $(m_0,n_0)\in B_{C}(0)\cap \mathbb{Z}\times \mathbb{Z}\subset \mathbb{R}^2$ such that $(k_p-pm_p,k_pq-n_pp)=(m_0,n_0)$ infinitely often and $\frac{\sqrt{m_0^2+n_0^2}}{2p}$ is the Euclidean injectivity radius of $T_{\frac{\pi}{4}}/\mathbb{Z}_p$ at $(z_0,w_0)$ for this subsequence. 
Hence, the sub-orbit generated by the translation  $(u,v) \rightarrow (u,v)+k_p(\frac{1}{p}, \frac{q}{p})$ is contained in the line $\mathbb{Z}\times \mathbb{Z} +\{t(m_0,n_0): t \in \mathbb{R}\}$. 
It follows that the $\text{Orbit}_{p,q}(z_0,w_0)$ is contained in  a union of equidistant lines parallel to the one described above by homogeneity.  
Modulo $\mathbb{Z}\times \mathbb{Z}$ the number of such lines is finite, let's denote it by $b_p$. Modulo $\mathbb{Z}\times \mathbb{Z}$  there are $\frac{p}{b_p}$ points of $\text{Orbit}_{p,q}(z_0,w_0)$ in each of these lines. Hence, 
\[\frac{p}{b_p}(\frac{k_p}{p},\frac{k_p q}{p})-\frac{p}{b_p}(m_p,n_p)=\frac{p}{b_p}(\frac{m_0}{p},\frac{n_0}{p}) \in \mathbb{Z}\times\mathbb{Z}.\] 
Therefore,  $b_p$ divides $m_0$ and  is independent of $p$. In other words, $\varphi_{\frac{\pi}{4}}(\text{Orbit}_{p,q}(z_0,w_0))$ is contained in $b$ integral curves of 
$X(z,w)=(m_0 \sqrt{-1} z, n_0 \sqrt{-1}w)\in \mathcal{X}(\mathbb{S}^3)$.
\end{proof}	

\begin{thm1.1}\label{main theorem}
\textit{There exists a positive  integer $p_0$ such that for every $p\geq p_0$ and every $q\geq 1$ the isoperimetric surfaces in  $L(p,q)$ are either geodesic spheres or tori of revolution about geodesics.}
\end{thm1.1} 

\begin{proof} [Proof of Theorem 1.1]
Arguing by contradiction, let's assume that there exists an infinite sequence of Lens spaces $L(p,q)$ containing an isoperimetric surface $\Sigma_p$ of genus $2$ or $3$ for each $p$.

We  use the Cheeger-Gromov convergence for the sequence of Lens spaces:

\noindent \textbf{Cheeger-Gromov Convergence}\,: A sequence $(M_i,g_i,p_i)$ converges, in the sense of Cheeger-Gromov,  to 
$(M,g,p)$ if the following two conditions hold true:
\begin{enumerate}
\item There exists an exhaustion of $M$ by compacts $\Omega_i$, i.e. $\Omega_i\subset\Omega_{n+1}$ and $\bigcup_i\Omega_i=M$.
\item There exists a family of diffeomorphism onto their images, $\phi_i:\Omega_i\rightarrow M_i$, such that $\phi_i(p)=p_i$ and $\phi_i^{*}g_i\rightarrow g$ in the
$C^{\infty}$ topology.
\end{enumerate} 
We consider  pointed manifolds  $(L(p,q), p^2 g_{\mathbb{S}^3},x_p)$ with  base points  $x_p$  belonging to $ \Sigma_p$. By the Cheeger-Gromov compactness theorem we have that $(L(p,q), p^2 g_{\mathbb{S}^3},x_p)\rightarrow(M,\delta, x_{\infty})$, where $(M,\delta)$ is a flat three manifold. The inclusion of $\Sigma_p$ into $M$ through the diffeomorphism $\phi_p$ is still denoted by $\Sigma_p$. 
  
\begin{lemma}\label{bound  second fundmental form lemma}
\textit{Let $(L(p,q), p^2 g_{\mathbb{S}^3},x_p)\rightarrow (M,\delta, x_{\infty})$ as above. There exists a constant $C>0$ such that $|A_{\Sigma_p}|_{p^2g_{\mathbb{S}^3}}\leq C$.}
\end{lemma}
\begin{proof}
Let $y_p\in \Sigma_p \subset L(p,q)$ be such that $|A_p|(y_p)=\max_{\Sigma_p}|A_p|^2$ and define
$\lambda_p= \max_{\Sigma_p}|A_p|(y_p)$. Arguing by contradiction, let's assume that $\frac{\lambda_p}{p}\rightarrow \infty$. In local coordinates around $y_p$  we consider the surface $\Sigma_p^{\prime}=\lambda_p \Sigma_p$ on the Euclidean ball $B_{\lambda_p\frac{\pi}{10p}}(0)$ endowed with the  rescaled  metric $\lambda_p^2\,g_{\mathbb{S}^3}$.
Therefore, $(B_{\lambda_p\frac{\pi}{10p}}(0),\lambda_p^2\,g_{\mathbb{S}^3}, y_p)$  converges to $(\mathbb{R}^3,\delta,0)$ as $p \rightarrow \infty$. The surface $\Sigma_p^{\prime}$  now has the property that $\max_{\Sigma_p^{\prime}}|A_p^{\prime}(0)|^2=1$.

By  the strong compactness  for a sequence of isoperimetric surfaces with bounded second fundamental form, see  Corollary \ref{compactness iso} in the Appendix,  there exists a subsequence converging to a properly  embedded surface $\Sigma^{\prime}\subset \mathbb{R}^3$, the convergence is in the sense of graphs and  with multiplicity one. Moreover,
$\Sigma^{\prime}$ is also stable, i.e.:
\[I_{\Sigma^{\prime}}(f,f)\geq 0, \,\, \forall f\in C_0^{\infty}(\Sigma^{\prime})\,\, \textit{satisfying}\, \int_{\Sigma^{\prime}}f\,d_{\Sigma^{\prime}}=0.\]
If $\Sigma^{\prime}$ is compact, then it has to be a round sphere by Alexandrov's Theorem, which is a contradiction since strong convergence preserves topology.  
If $\Sigma^{\prime}$ is
non-compact, then it has infinite area  by the monotonicity formula:
indeed, by Proposition \ref{monotonicity} in the Appendix there exists a positive constant $C$ such that
\[\frac{d}{dr}\bigg(\frac{e^{C\,r}\,|\Sigma^{\prime}\cap B_r(x)|}{r^2}\bigg)\geq 0.\] In particular, $|\Sigma^{\prime}\cap B_r(x)|\geq \pi r^2$. As $\Sigma^{\prime}$ is properly embedded,  it has infinite extrinsic diameter and the claim follows. Therefore,  $\Sigma^{\prime}$ is totally geodesic by Da Silveira's Theorem \ref{silveira} in the Appendix, which is a contradiction since $\max_{\Sigma^{\prime}}|A|=1$.
\end{proof}
\begin{lemma}\label{rank one}
\textit{Let $(L(p,q), p^2 g_{\mathbb{S}^3},x_p)\rightarrow (M,\delta, x_{\infty})$ as above, then  $M$ is a flat manifold of rank at most one.}
\end{lemma}
\begin{proof}
Below we denote $T_r/\mathbb{Z}_p$ by $T_r$. Let $T_{r_p}$ be the Clifford torus through $x_p$ enclosing a region $\Omega_{r_p}$. Under the scaling by $\lambda_p=p^2$ we have that $|\Omega_{r_p}|= 2\pi^2 p^2 \sin^2(r_p)$.  If $\lim_{p\rightarrow\infty}|\Omega_{r_p}|<\infty$, then $\lim_{p\rightarrow \infty}|T_{r_p}|= 2\pi^2 p\sin(2r_p)<\infty$. Moreover, the second fundamental form $A_{r_p}$   of $T_{r_p}$ satisfies $\lim_{p\rightarrow \infty} |A_{r_p}|^2=\lim_{p\rightarrow \infty} \frac{1}{p^2}(\frac{\cos^2(r_p)}{\sin^2(r_p)}+ \frac{\sin^2(r_p)}{\cos^2(r_p)})<\infty$. The critical fiber $T_0\subset \Omega_{r_p}$ is distant from $x_p$ by $O(\frac{1}{p})$ since $\Omega_{r_p}$ is converging to a compact region in $M$. Instead of using base points $x_p$ we choose new base points $y_p \in T_0$; it follows that $(L(p,q),p^2g_{\mathbb{S}^3},y_p)\rightarrow (N,\delta,y_{\infty})$  and $\text{rank}(N)=\text{rank}(M)$.  
We claim that rank of $N$ is at most one:
\[|B_{2R}(y_{\infty})|=\lim_{p\rightarrow \infty} |B_{2R}^p(y_p)|\geq \lim_{p\rightarrow \infty}|\Omega_{\frac{R}{p}}|=\lim_{p \rightarrow \infty} 2\pi^2 p^2 \sin^2(\frac{R}{p})= c R^2.\]
Let's assume now that $\lim_{p\rightarrow \infty}|\Omega_{r_p}|=\infty$, consequently $\lim_{p\rightarrow \infty}|T_{r_p}|=\infty$ and $\lim_{p\rightarrow \infty}|A_{r_p}|^2=0$. 
Recall the function $r=r(x)$, the distance from the Clifford torus through $x$ to the  critical fiber $T_0$ with respect to the round metric. The unit vector field $\partial r$ is  orthogonal to $T_r$ for every $r$ and it is well  defined on $L(p,q)-\{T_0\cup T_{\frac{\pi}{2}}\}$. Let $\gamma(r)$ be the geodesic whose velocity is $\partial r$ and such that $\gamma(r_p)=x_p$. Consider $K_{r_p,R}=\{ x \in T_r: d_{L(p,q)}(x,\gamma(r))\leq \frac{R}{p} \, \text{and}\, |r-r_p|\leq \frac{R}{p}\}$. By the triangle inequality $K_{r_p,R}\subset B_{2R}(x_p)$ under the metric $p^2g_{\mathbb{S}^3}$. Applying the coarea formula for $f(r)=p\,r$, $|\nabla f|_{p^2g_{\mathbb{S}^3}}=1$, we obtain:
\begin{eqnarray*}|K_{r_p,R}|
=\int_{r_p-\frac{R}{p}}^{r_p+\frac{R}{p}}|B_R(\gamma(u))\cap T_u|p\,du 
=|B_R(\gamma(u_0))\cap T_{u_0}|_{p^2g_{\mathbb{S}^3}}R \geq c R^2,
\end{eqnarray*}
where $u_0\in [r_p- \frac{R}{p}, r_p + \frac{R}{p}]$ is from the mean value theorem for integrals. 
The last inequality is justified as follows. Either the extrinsic diameter of $T_{u_0}$ is going to infinity and  $T_{u_0}$ is converging with multiplicity to a flat surface or the extrinsic diameter of $T_{u_0}$ is bounded. The former implies that $|B_R(\gamma(u_0))\cap T_{u_0}|_{p^2g_{\mathbb{S}^3}}\geq c R$. The latter  implies that  $B_R(\gamma(u_0))\cap T_{u_0}=T_{u_0}$, which is a contradiction since $|T_{u_0}|_{p^2g_{\mathbb{S}^3}}\rightarrow \infty$. We have conclude that $Vol(B_{2R})\geq c R^2$ and rank of $M$ is at most one.
\end{proof} 
 
The following lemma gives a description of $I_{L(p,q)}$ for small volumes:

\begin{lemma}\label{isoperimetric small volume}
\textit{For $p$ large enough there exist $v_p$ and $\varepsilon_p>0$ such that $I_{L(p,q)}$ is given by the profile of spheres on $(0,v_p]$ and by the profile of flat tori on $[v_p,v_p+\varepsilon_p)$. Moreover, if $\Sigma_p$ is an isoperimetric surface such that $I_{L(p,q)}(v_p)=|\Sigma_p|$, then $g(\Sigma_p)=0\, \text{or}\, 1$.}
\end{lemma}
\begin{proof} 
For each $p$ we consider the first volume, $v_p$, for which there is  transition on topology of isoperimetric surfaces from spheres to something else. 
If $v_*$ is the volume for which the profile of geodesic spheres intersect the profile of flat tori, then  $v_p\leq v_*$.  The value of $v_*$ is computed by solving the following system of equations:
\begin{eqnarray*}
 \frac{2\pi^2}{p} \sin^2(r)= 2\pi s - \pi \sin (2s)\quad \text{and}\quad \frac{2\pi^2}{p} \sin(2r)=4\pi \sin^2(s). 
\end{eqnarray*}
The  left hand sides (right hand sides) correspond to the enclosed volume and area of the Clifford torus $T_r$ (geodesic spheres $S_s$ of radius $s$), respectively.
It follows that  $s\leq \frac{\pi}{p}$; another way to see this is by recalling that the injectivity radius of $L(p,1)$ is $\frac{\pi}{p}$ at every point. Therefore, $I_{L(p,q)}(v_p)\leq O(\frac{1}{p^2})$. 

Let $\Sigma_p$ be  an isoperimetric surface  with the property that $I_{L(p,q)}(v_p)=|\Sigma_p|$.  By Lemma \ref{bound  second fundmental form lemma} the sequence $\{\Sigma_p\}_{p\in\mathbb{N}}$ has bounded second fundamental form in $(L(p,q), p^2 g_{\mathbb{S}^3},x_p)$; thus, it strongly converges to a  properly embedded surface $\Sigma$ of finite area in some orientable flat three manifold $(M,\delta)$ of rank at most one by Lemma \ref{rank one}. By the monotonicity formula, Proposition  \ref{monotonicity} in the Appendix, $\Sigma$  is a closed  surface. 
It follows that the pre-image $\widehat{\Sigma}$ of $\Sigma$ in $\mathbb{R}^3$ is contained in a solid cylinder. Hence,  $\widehat{\Sigma}$ is an union of round spheres by Alexandrov's Theorem or is a surface of revolution about the axis of the cylinder by Theorem \ref{kks} in the Appendix.
Therefore,   $g(\Sigma)=0$ or $1$.
From the strong  compactness for isoperimetric surfaces, see Corollary \ref{compactness iso}, we have that 
$g(\Sigma_p)=0\, \text{or}\, 1$, $v_p=v_{*}$ and the existence of the desired $\varepsilon_p>0$.
\end{proof}

\begin{claim}
\textit{Theorem \ref*{theorem 1} follows if we can show that the isoperimetric surfaces separating $L(p,q)$ in two regions of the same volume are tori.}
\end{claim} 
\begin{proof}
By the strong compactness for isoperimetric surfaces, Corollary \ref{compactness iso}, there exists $\widehat{v}$ such that if $\Sigma$ is an isoperimetric surface enclosing volume $v \in [\widehat{v}, \frac{\pi^2}{p}]$, then $\Sigma$ is a flat torus.
It follows from   Lemma \ref{isoperimetric small volume}   and for large $p$ that the isoperimetric profile $I_{L(p,q)}$   is  given by the area of geodesic spheres for volumes in  $(0,v_p]$ and  by the area of flat tori for volumes in $[v_p,\varepsilon_p]\cup [\widehat{v},\frac{\pi^2}{p}]$. In other words, if
 $f(v)$ is the function defined by $f(v)=|T_{r(v)}|$, where
$T_{r(v)}$ is the Clifford torus enclosing a volume equal to $v$, then $I_{L(p,q)}(v)=f(v)$  on $[v_p,\varepsilon_p]\cup [\widehat{v},\frac{\pi^2}{p}]$.
It follows that $\phi(v)=f(v)-I_{L(p,q)}(v)$ has a   local maximum point  at $t_*\in [\varepsilon_p,\widehat{v}]$. 

The claim will follow by exploring the weak differential equation for $I_{L(p,q)}$.  From (\ref{second derivative area}) we have 
\[f^2(v)f^{\prime \prime}(v)+ \int_{T_{r(v)}}(2+ |A_{r(v)}|^2)\,d_{T_{r(v)}}=0.\] 
Let $\Sigma$ be an isoperimetric surface  such that $I_{L(p,q)}(t_*)=|\Sigma|$. If $\Sigma_v$ is the unit normal variation of $\Sigma$ parametrized by the enclosed volume, then we define  $h(v)= \textit{Area}(\Sigma_v)$. Since $h\geq I_{L(p,q)}$, we have that  $\phi_1= f- h$  has also  a local  maximum point at $t_*$. Hence,  $\phi_1^{\prime}(t_*)=0$, i.e., $H_r=H$, and $\phi_1^{\prime\prime}(t_*)\leq 0$. Applying equation (\ref{second derivative area}) for $h$ together with the Gauss equation and the Gauss-Bonnet theorem we obtain:
\begin{eqnarray*}
\phi_1^{\prime\prime}(t_*)\leq 0 &\Rightarrow & \frac{1}{h^2}\int_{\Sigma}(2+|A|^2)\,d_{\Sigma}\leq \frac{1}{f^2}\int_{T_{r(v)}}(2+ |A_{r(v)}|^2)\,d_{T_{r(v)}}\\
&\Rightarrow& 4 (1+H^2)\,h + 8\pi\,(g-1)\,\leq\, \big(4(1+H_r^2)\,f\big)\,\frac{h^2}{f^2}\\
&\Rightarrow&  1+H^2 + \frac{2\pi(g-1)}{h}\leq (1+H_r^2)\,\frac{h}{f}\leq 1+H^2.
\end{eqnarray*}
Therefore, $g(\Sigma)=1$ and $I_{L(p,q)}=f$ in $[v_p, \frac{\pi^2}{p}]$.
\end{proof} 

By Lemma \ref{subsequences}, the proof of Theorem \ref*{theorem 1} reduces to investigating the Cases I and II below.

\begin{flushleft} 
Case I: There is a subsequence  whose $\mathbb{Z}_{p}^q$ orbit of a point  is contained in a finite number (independent of $p$) integral curves of  a  vector field  $X(z,w)=(m_0\,\sqrt{-1}\,z, n_0\,\sqrt{-1}\,w)\in \mathcal{X}(\mathbb{S}^3)$.
\end{flushleft}
We claim that the injectivity radius of $L(p,q)$ at every point is $O(\frac{1}{p})$. Indeed,  let  $\text{Orbit}_{p,q}(z_0,w_0)$ be the orbit of $(z_0,w_0)\in \mathbb{S}^3$ with respect to $\mathbb{Z}_p^q$. As before,  $\frac{p}{b}$ points of $\text{Orbit}_{p,q}(z_0,w_0)$  lie on  the curve $\beta(t)=\psi(z_0,w_0,t)$.
Here, $\psi$ is the one parameter family  of diffeomorphisms associated to $X$:
\begin{eqnarray*}\psi: \mathbb{S}^3\times \mathbb{R} \rightarrow \mathbb{S}^3 :
\psi(z,w,t)=(e^{m_0 i t}z,e^{n_0 i t}w).\end{eqnarray*}
When ordered according the orientation of $\beta(t)$, those points determine a piecewise closed geodesic $\gamma_p(t)$ with $\gamma_p(0)=(z_0,w_0)$. As $p \rightarrow \infty$,  $\gamma_p(t)$ converges to $\beta(t)$. The claim now follows from :
\begin{equation}\label{injective radius}
 \lim_{p \rightarrow \infty}\frac{p}{b} inj_{(z_0,w_0)} L(p,q) \leq \lim_{p \rightarrow \infty}\mathcal{H}^1(\gamma_p)= 2\pi \sqrt{m_0^2|z_0|^2+n_0^2|w_0|^2}.
\end{equation}
 
By the Cheeger-Gromov Compactness Theorem we obtain, up to subsequence, the following convergence: 
\begin{equation}\label{convergence}
(L(p,q), \,p^2 \cdot g_{\mathbb{S}^3},\, x_p) \xrightarrow{C-G} (M,\, \delta,\, x_{\infty}),
\end{equation}
where $(M,\,\delta, \, x_{\infty})$ is a flat three manifold of rank one by Lemma \ref{rank one}.  
By (\ref{injective radius}) we conclude that $M=\mathbb{S}^1(r_0)\times \mathbb{R}^2=\mathbb{R}^3/\mathbb{Z}$. 

The  curves $t\rightarrow \beta(t)=\psi(x,t)$ represent  integral curves of $X$  through  $x \in L(p,q)$; they have bounded geodesic curvature  and    $\mathcal{H}^1(\beta)=O(\frac{1}{p})$.
Hence, the integral curves of $X$ converge to closed geodesics in $M$ under (\ref{convergence}). As sets, they  coincide with the standard fibers of $M=\mathbb{S}^1\times \mathbb{R}^2$.

By the Poincare-Hopf index theorem  there exists a zero  for the vector field $\frac{X^{\top}}{|X|}\in \mathcal{X}(\Sigma_p)$ since $g(\Sigma_p)\geq 2$. Hence, we can choose the base points $x_p$  to satisfy  $g_{\mathbb{S}^3}(\frac{X}{|X|}(x_p), N(x_p))=\pm 1$, here $N$ is the unit normal vector of $\Sigma_p$.  
\begin{lemma}
There exists  a properly embedded surface $\Sigma_{\infty} \subset M$ such that $(\Sigma_p, x_p)\rightarrow (\Sigma_{\infty},x_{\infty})$ with multiplicity one. Moreover, $\Sigma_{\infty}$ is totally geodesic and perpendicular to the standard fibers of $\mathbb{S}^1\times \mathbb{R}^2$.
\end{lemma}
\begin{proof}
By  Lemma \ref{bound  second fundmental form lemma} the sequence of isoperimetric surfaces $(\Sigma_p, x_p)\subset (L(p,q),p^2 g_{\mathbb{S}^3},x_p)$ has uniformly bounded second fundamental form.
Applying the strong compactness theorem for isoperimetric surfaces,  Corollary \ref{compactness iso},  we conclude that $\Sigma_p$
converges smoothly and with multiplicity one to a properly embedded  stable CMC surface $\Sigma_{\infty}\subset \mathbb{S}^1\times \mathbb{R}^2$.
If $\liminf_{p\rightarrow \infty} |\Sigma_p|_{p^2g_{\mathbb{S}^3}}<\infty$,
then the monotonicity formula, Proposition \ref{monotonicity}, implies that the extrinsic diameter of $\Sigma_p$ and $\Sigma_{\infty}$ are bounded.  This is impossible since the sequence $\Sigma_p$ separates $L(p,q)$ in two regions of the same volume that goes to infinity as $p\rightarrow\infty$.
Therefore,  $\Sigma_{\infty}$ is a complete properly embedded stable CMC surface in $\mathbb{S}^1\times \mathbb{R}^2$ with infinite area.
Applying Da Silveira's Theorem \ref{silveira} once more, we obtain that $\Sigma_{\infty}$ is totally geodesic.
As $g_{\mathbb{S}^3}(\frac{X}{|X|}(x_p), N(x_p))=\pm 1$, we conclude that $\Sigma_{\infty}$ is orthogonal to the standard fibers of $\mathbb{S}^1\times \mathbb{R}^2$.
\end{proof}
We claim that $\Sigma_{\infty}$ separates
$\mathbb{S}^1\times \mathbb{R}^2$. If it does not separate, then there exists a loop $\gamma$ intersecting $\Sigma_{\infty}$  at a single point. As $\Sigma_p \rightarrow \Sigma_{\infty}$ with multiplicity one,  the same conclusion holds for $\Sigma_p$, which  contradicts  the fact that $\Sigma_p$ separates $L(p,q)$. Therefore, there is $k\geq 1$ such that  $\Sigma_{\infty}= \partial \Omega_{\infty}=\bigcup_{i=1}^{2k} \sigma_i$, where $\sigma_i$ is a flat plane for each $i$. 
\begin{claim}
\textit{This configuration cannot be a limit of isoperimetric surfaces.}
\end{claim} 
We regard   $\mathbb{S}^1\times \mathbb{R}^2$  as a slab in $\mathbb{R}^3$ with height $2\pi$. Now, we construct a deformation of $\Sigma_{\infty}$ which decreases its area as follows. First, we cut off  the $k$ solid cylinders obtained from the intersection of  $\Omega_{\infty}$ with a  vertical solid torus of radius $R$. 
To balance the enclosed volume we add  a vertical solid torus of radius $r$, see Figure 1  for the case $k=1$. If $a_i$ is the distance between $\sigma_{2i}$ and $\sigma_{2i-1}$, then the radius $r$ is given by:
\[\sum_{i=1}^k\pi\,R^2a_i=\pi\, r^2 2\pi \Rightarrow r= R\sqrt{\frac{\sum_i a_i}{2\pi}}.\]
The boundary of this new region is denoted by $\widetilde{\Sigma}_{\infty}$ and
\[Area(\widetilde{\Sigma}_{\infty}\cap K)=Area(\Sigma_{\infty}\cap K)- 2k\pi\,R^2 + 2\pi\,R \cdot \sum_i a_i+ 2\pi\,r\cdot 2\pi.\]
\begin{figure}[h]
\begin{center}
  \includegraphics[scale=0.5]{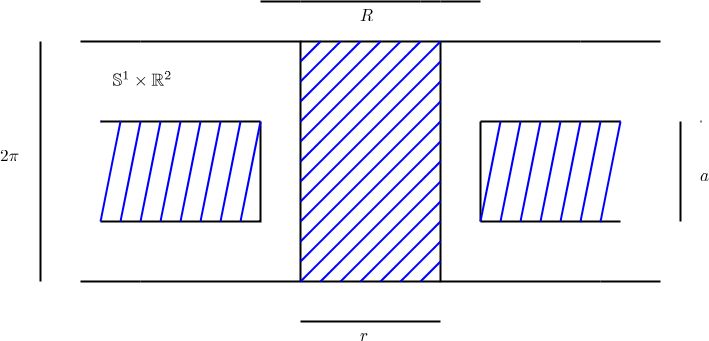}\\
  \caption{Compact support deformation of $\Sigma_{\infty}$.}
  \end{center}
\end{figure}
If $R$ is large enough, then $\widetilde{\Sigma}_{\infty}$ has less area than $\Sigma_{\infty}$. This is impossible since the strong multiplicity one convergence  allow us to carry out this deformation of $\Sigma_{\infty}$ to  $\Sigma_p$  which contradicts the fact that $\Sigma_p$ is an isoperimetric surface.
\begin{flushleft}
Case II: There is a subsequence $p\rightarrow \infty$ where the $\mathbb{Z}_{p}^q$ orbit of a point is becoming dense on the Clifford torus  containing  such point.
\end{flushleft}
We  use geometric measure theory methods to analyse  the pre-image sequence $\{\Sigma_p^{\prime}\}_{p\in \mathbb{N}} \subset \mathbb{S}^3$.  

It is proved in \cite{R2} that if $\Sigma=\partial \Omega$ has positive mean curvature, then
\[3\, |\Omega|\,\leq\, \int_{\Sigma} \frac{1}{H}\,d_{\Sigma}.\]
Applying this formula to the sequence of isoperimetric surfaces, we conclude that  the mean curvature of $\Sigma_p$ satisfy $H_p\leq \frac{|\Sigma_p|}{3|\Omega_p|}\leq \frac{2}{3}$. We  have used  that $\Sigma_p$ separates $L(p,q)$ into  regions of the same volume.

As $\{\Sigma_p\}_{p\in\mathbb{S}^3}$ has area and mean curvature bounded, we  apply \textit{Allard's compactness theorem}, Theorem 42.7 and Remark 42.8 in \cite{S},   to obtain an integral varifold $0\neq\mathcal{V}^2\subset \mathbb{S}^3$ that, up to subsequence, is $\mathcal{V}^2=\lim_{p\rightarrow \infty} \Sigma_p^{\prime}$.
Recall that $\mathbb{S}^1\times \mathbb{S}^1$ acts on $\mathbb{S}^3$ via $(z,w)\rightarrow(\alpha_1,\alpha_2)(z,w):=(\alpha_1\,z,\alpha_2\,w)$. We claim that $\mathcal{V}^2$ is $\mathbb{S}^1\times \mathbb{S}^1$ invariant:
indeed, if $(z,w)\in \text{supp}(\mathcal{V}^2)$ and $(\alpha_1,\alpha_2) \in \mathbb{S}^1\times \mathbb{S}^1$, then for each $p$ there is $l_p\in \mathbb{Z}$ such that $\lim_{p \rightarrow \infty} (e^{\frac{2\pi i l_p q}{p}}z, e^{\frac{2\pi i l_p}{p}}w)= (\alpha_1z,\alpha_2w)$. On the other hand,  as $(z,w)\in \mathcal{V}^2$,  there is 
$(z_p,w_p)\in \Sigma_p^{\prime}$  such that $(e^{\frac{2\pi i l_p q}{p}}z_p, e^{\frac{2\pi i l_p}{p}}w_p) \in \Sigma_p^{\prime}$ and $\lim_{p\rightarrow \infty}(z_p,w_p)=(z,w)$.   It follows that $\lim_{p \rightarrow \infty}(e^{\frac{2\pi i l_p q}{p}}z_p, e^{\frac{2\pi i l_p}{p}}w_p)= (\alpha_1 z,\alpha_2 w)$ and $(\alpha_1,\alpha_2)(z,w)\in \text{supp}(\mathcal{V}^2)$. In particular, $supp\,\mathcal{V}^2 = \bigcup_{j=1}^k T_{r_j}$ where $T_{r_j}$ is a  Clifford torus. 
The  monotonicity formula implies that the convergence $\Sigma_p^{\prime}\rightarrow \mathcal{V}^2$ is also in Hausdorff distance; hence, we have that $\Sigma_p^{\prime}=\cup_{j=1}^k \Sigma_p^{\prime j}$ and $ \text{supp}\, (\lim_{p\rightarrow \infty}\Sigma_p^{\prime j})= T_{r_j}$.  Since $\Sigma_p^{\prime}$ is $\mathbb{Z}_{p}^q$ invariant there exists  $\theta_p^{j_1 j_2} \in \mathbb{Z}_p^q\subset \mathbb{S}^1\times\mathbb{S}^1$  for which $\theta_p^{j_1 j_2} (\Sigma_i^{\prime j_2})=\Sigma_p^{\prime j_1}$. By taking the limit we obtain  $(\alpha_1,\alpha_2)(T_{r_{j_2}})=T_{r_{j_1}}$ for some $(\alpha_1,\alpha_2) \in \mathbb{S}^1\times\mathbb{S}^1$. As this is impossible, we conclude that $k=1$ and all $\Sigma_p^{\prime}$ are connected for $p$ large.

Now we consider $\{\Sigma_p^{\prime}\}_{p=1}\subset \textbf{I}_2(\mathbb{S}^3, \mathbb{Z})$, the space of   2-dimensional integral \textit{currents} on $\mathbb{S}^3$. Each $\Sigma_p^{\prime}=\partial \Omega_p^{\prime}$ and $\Omega_p^{\prime}\in \textbf{I}_3(\mathbb{S}^3, \mathbb{Z})$. As $\Omega_p^{\prime}$ is a region of finite perimeter ($\mathcal{X}_{\Omega_p^{\prime}}$ is BV function with uniform bounded variation), then $\Omega_i^{\prime}\rightarrow \Omega^{\prime}$ and  $\Sigma_p^{\prime} \rightarrow \partial \Omega^{\prime}$ as currents, $\Omega^{\prime}$ is an open set of finite perimeter, see Theorem 6.3 and proof of Theorem 37.2   in \cite{S}. 
Since $|\Omega_p|=\pi^2$ we conclude that $|\Omega^{\prime}|=\pi^2$.  
Applying  the \textit{Constancy theorem}, Theorem 26.27 in \cite{S}, we conclude that  $\Omega^{\prime}$ is the handlebody bounded by the Clifford torus $T_{r_1}$, and consequently $r_1= \frac{\pi}{4}$. 

We proved that $\mathcal{V}^2=m\,T_{\frac{\pi}{4}}$ for some positive integer $m \in \mathbb{N}$.
Since $\Sigma_p$ is isoperimetric, it follows that  $m=1$. Indeed,
\[ m\,|T_{\frac{\pi}{4}}|=|\mathcal{V}^2|=\lim_{p\rightarrow \infty}|\Sigma_p^{\prime}| \leq  |T_{\frac{\pi}{4}}|.\] 
As $T_{\frac{\pi}{4}}$ is smooth, we have for  $r>0$ sufficiently small  that the density $\theta(T_{\frac{\pi}{4}},r,x)\leq 1+\frac{\epsilon}{2}$, where $\epsilon>0$ is from Theorem \ref{Allard regularity} in the Appendix. On the other hand, as
$\Sigma_p^{\prime}$ is converging to $T_{\frac{\pi}{4}}$ with multiplicity one, then   $\theta(\Sigma_p^{\prime},x,r)\leq 1+\epsilon$ for $p$ large enough.
Now we invoke the smooth version of \textit{Allard's Regularity Theorem}, Theorem \ref{Allard regularity}, to concluded that 
the convergence $\Sigma_p^{\prime}\rightarrow T_{\frac{\pi}{4}}$ is strong, i.e.,  graphical with multiplicity one. 
As  strong convergence  preserves topology, we conclude that $g(\Sigma_p)=1$. This completes the  proof of Theorem \ref*{theorem 1}.
\end{proof}

\subsection{Berger Spheres}
Let $g_0$ be the round metric on $\mathbb{S}^3$ and $J$  the vector field on $\mathbb{S}^3$ defined as $J(z,w)=(\sqrt{-1}\, z,\sqrt{-1}\, w)$. Recall that $J$ is tangent to the fibers of the Hopf fibration $h:\mathbb{S}^3\rightarrow \mathbb{S}^2(\frac{1}{2})$.

The \textit{Berger metrics} are Riemannian metrics $g_{\varepsilon}$   on $\mathbb{S}^3$ defined as:
\[g_{\varepsilon}(X,Y)=g_0(X,Y) + (\varepsilon^2 -1)g_0(X,J)g_0(Y,J),\,\,\, \varepsilon >0.\]
The Riemannian manifolds $(\mathbb{S}^3,g_{\varepsilon})$ are  called the \textit{Berger spheres}, they are denoted by $\mathbb{S}_{\varepsilon}^3$. 
Geometrically, the metric $g_{\varepsilon}$ shrinks the Hopf fibers to have length $2\pi\,\varepsilon$.

The Berger metrics  are   also homogeneous   and their  group of isometries   has  dimension four.  
It follows from the work of  Abresch and Rosemberg \cite{AR}   that every constant mean curvature surface  in $\mathbb{S}_{\varepsilon}^3$ admits a holomorphic quadratic differential. In particular, every CMC sphere in $\mathbb{S}_{\varepsilon}^3$ is rotationally invariant.

A  precise study of closed orientable surfaces with constant mean curvature on the Berger spheres is given in 
\cite{TU}. It is  proved there  the existence of  $\varepsilon_1>0$ with the following property: if $\varepsilon\in[\varepsilon_1,1]$, then every  stable constant mean curvature surface  in $\mathbb{S}_{\varepsilon}^3$  has  
genus zero or one. Moreover,   if $\varepsilon^2 \in [\frac{1}{3},1]$, then  these stable CMC surfaces are  totally umbilical spheres or  the  minimal Clifford torus, the latter  only occurring  when $\varepsilon^2=\frac{1}{3}$. In particular,  rotationally invariant spheres are the only solutions of the isoperimetric problem in $\mathbb{S}_{\varepsilon}^3$ for $\varepsilon^2 \in [\frac{1}{3},1]$.

\begin{theorem}
\textit{There exists $\varepsilon_0\,>\,0$ such that for every $\varepsilon<\varepsilon_0$ the isoperimetric surfaces in the Berger spheres $\mathbb{S}_{\varepsilon}^3$ are either rotationally invariant spheres or  tori.} 
\end{theorem}

\begin{proof}
Let's assume that for every $\varepsilon$ there exists an isoperimetric surface $\Sigma_{\varepsilon}$ with $g(\Sigma_{\varepsilon})=2$ or $3$.

We rescale  the metric $g_{\varepsilon}$ of  $\mathbb{S}_{\varepsilon}^3$ by the factor $\lambda_{\varepsilon}=\frac{1}{\varepsilon^2}$. The Hopf fibers have  constant length equal to $2\pi$ under the new metric $\lambda_{\varepsilon}g_{\varepsilon}$. It follows that the injective radius of $\mathbb{S}_{\varepsilon}^3$ at a point $p$ is equal to  $inj_p\mathbb{S}_{\varepsilon}^3=\pi$ for every $p\in\mathbb{S}_{\varepsilon}^3$.  

Since $h$ is a local trivial fibration,  we have that  for each $p_{\varepsilon}\in \mathbb{S}_{\varepsilon}^3$  there exist a neighbourhood $V$ of $h(p_{\varepsilon})$ and a diffeomorphism $\phi_{\varepsilon}: V\times \mathbb{S}^1 \rightarrow h^{-1}(V)$ such that $h\circ\phi_{\varepsilon}=\pi_1$, where $\pi_1: V \times \mathbb{S}^1\rightarrow V$ given by $\pi(x,y)=x$. Moreover, $\phi_{\varepsilon}^{*}(\frac{1}{\varepsilon^2}g_{\varepsilon})\rightarrow \delta$ in the $C^{\infty}$ topology. Therefore, in the sense of Cheeger-Gromov we have:
\[\big(\mathbb{S}_{\varepsilon}^3,\frac{1}{\varepsilon^2}\,g_{\varepsilon},p_{\varepsilon}\big) \rightarrow \big(\mathbb{S}^1\times\mathbb{R}^2,\delta,0\big).\]
 
We pick the  points $p_{\varepsilon}\in \Sigma_{\varepsilon}$  with the property that
$g_{\varepsilon}(J,N_{\varepsilon})(p_{\varepsilon})=\pm\,\varepsilon$, this means $J$ and $N_{\varepsilon}$ are parallel at $p_{\varepsilon}$. These points exist by the Poincar\'{e}-Hopf index theorem. By Lemma \ref{bound  second fundmental form lemma}   the inclusion of $\Sigma_{\varepsilon}$ in $\big(\mathbb{S}^1\times \mathbb{R}^2, \phi_{\varepsilon}^{*}(\varepsilon^{-2}g_{\varepsilon})\big)$  has the following property:
\[\textit{There exists}\,\,\,C>0\,\, \textit{such that}\,\,\, \sup_{\Sigma_{\varepsilon}}|A_{\varepsilon}|\leq C\,\,\, \text{for every}\,\varepsilon.\]
By the strong compactness theorem for  isoperimetric surfaces, Corollary \ref{compactness iso} in the Appendix,  we can extract a subsequence, $\{\Sigma_{\varepsilon_n}\}$, which converges  with multiplicity one  to a properly embedded surface
$\Sigma_{\infty}\subset (\mathbb{S}^1\times \mathbb{R}^2,\delta).$

If  $Area(\Sigma_{\infty})<\infty$, then   the monotonicity formula,  Proposition \ref{monotonicity}, implies that $\Sigma_{\infty}$  is compact. We apply
Theorem  \ref{kks} to conclude that $\Sigma_{\infty}$ is either a round sphere or torus. This is impossible since we have strong convergence and $g(\Sigma_{\varepsilon})=2\,or\, 3$.
Therefore,  $\Sigma_{\infty}$ is a complete non-compact surface with infinite area. Moreover, 
$\Sigma_{\infty}$ is also a stable CMC surface in $\mathbb{S}^1\times \mathbb{R}^2$:
\[I_{\Sigma_{\infty}}(f,f)\geq 0, \,\,\, \forall f\in C_0^{\infty}(\Sigma_{\infty})\,\, \textit{such that}\,\, \int_{\Sigma_{\infty}}f\,d_{\Sigma_{\infty}}=0.\]
It follows from Theorem \ref{silveira} that $\Sigma_{\infty}$ is   totally geodesic.
By the choice of $p_{\varepsilon}$ we conclude that    $\Sigma_{\infty}$ is orthogonal to the $\mathbb{S}^1$ fibers of $\mathbb{S}^1\times \mathbb{R}^2$. Since $\Sigma_{\infty}$ separates $\mathbb{S}^1\times \mathbb{R}^2$, we also conclude that $\Sigma_{\infty}$ is an union of at least two totally geodesic planes. As shown in the proof of Theorem 1.1, this configuration cannot be a limit of isoperimetric surfaces.
\end{proof}
\begin{section}{Appendix}

In this Appendix we collect some background results for surfaces with constant mean curvature in 3-manifolds.
\begin{proposition}\label{monotonicity}
\textit{Let $M^3$ be  a three manifold with bounded curvature,  $|K_M|\leq k$, and with positive lower bound on the injectivity radius $inj(M)\geq i_0$. If $\Sigma \subset M^3$ is a smooth surface with mean curvature which satisfies $|H|\leq H_0$, then there exists a  positive constant $C=C(H_0, k)$ such that
\begin{eqnarray}
\frac{d}{dr}\bigg(\frac{e^{C\,r}\,Area(\Sigma\cap B_r(p))}{r^2}\bigg)\geq 0,
\end{eqnarray}
for every $p\in \Sigma$ and $r\leq \min\{i_0,\frac{1}{\sqrt{k}}\}$.}
\end{proposition}
\begin{proof}
See Chapter 7 in \cite{CM}.
\end{proof}
Let $\Sigma$ be a CMC surface in a closed manifold $M^3$. The density of $\Sigma$ at $x$ is given by 
\[\theta(\Sigma,x,r)= \frac{\text{Area}(\Sigma\cap B_r(x))}{\pi\,r^2}.\]
\begin{theorem}[Allard's Regularity Theorem]\label{Allard regularity}
\textit{Let $M^3$ be a closed manifold and $\rho>0$. There exist $\epsilon=\epsilon(M,\rho)>0$ and $C=C(M,\rho)$ with the following property: if $\Sigma \subset M$ is a smooth embedded  CMC surface satisfying 
\[\theta(\Sigma, x, r)\leq 1+ \epsilon\]
for every $x \in M$ and $r< \rho$, then its second fundamental form is uniformly bounded, i.e., $|A_{\Sigma}|\leq C$.}
\end{theorem}
\begin{proof}
See Theorem 1.1 in \cite{W}.
\end{proof}

Let $\{\Sigma_n\}_{n\in\mathbb{N}}$ be a sequence of surfaces in a manifold $M$. We say that $\Sigma_n$  converge to  $\Sigma$ in the sense of graphs if near any point $p\in\Sigma$ and for large $n$  the surface $\Sigma_n$ is locally a  graph over an open set of $T_p\Sigma$ and these graphs converge smoothly to the graph of $\Sigma$.
In addition, we say that $\{\Sigma_n\}$ satisfy \textit{local area bounds} if there exist $r>0$ and $C>0$ such that $|\Sigma_n\cap B_r(x)|\leq C$ for every $x\in M$.

A hypersurface $\Sigma$ is said to be \textit{weakly embedded}  if it admits only tangential self intersections. 
\begin{proposition}\label{compactness}
\textit{Let $\{\Sigma_n\}\subset(M,g_n) $ be a sequence of embedded surfaces  with constant mean curvature  satisfying local area bounds and such that $\sup_{\Sigma_n}|A_n|\leq C$. Let's assume that $g_n$ converges to a  metric $\delta$ in the $C^{\infty}$ topology. If $\{\Sigma_n\}_{n=1}$ has an accumulation point,  then we can extract a  subsequence that converges   to a properly weakly embedded CMC surface $\Sigma$ in $(M,\delta)$.}
\end{proposition}
\begin{proof}[Sketch of the Proof.]
Let's first recall the constant mean curvature equation for graphs. If $\Sigma^{\prime}$ is a   surface with constant mean curvature $H$ in $(M,g)$, then  $\Sigma^{\prime}$ can be written locally as a graph over a neighbourhood $U_p \subset T_p\Sigma^{\prime}$:
\[\Sigma^{\prime}= Graph(\,u\,)=\{(x_1,x_2,u(x_1,x_2): x_1,x_2 \in U_p)\}.\]
In coordinates  $g_{ij}:=g(e_i,e_j)$ where $\{e_1,e_2,e_3\}$ is the coordinate base associated to $(x_1,x_2,x_3)$. Let $\{E_1,E_2\}$ be the coordinate base for $\Sigma_n$, i.e., $E_i=e_i+u_{x_i}e_3=T_i^le_l$. The induced metric $h$ is expressed  by $h_{ij}=h(E_i,E_j)$.
A simple computation gives:
\[g(N,e_i)=\frac{-u_{x_i}}{\sqrt{1+g^{ij}u_{x_i}u_{x_j}}}\,\,\,\text{and}\,\,\, g(N,e_3)=\frac{1}{\sqrt{1+g^{ij}u_{x_i}u_{x_j}}}.\]
We also have
\begin{eqnarray*}
\nabla_{E_i}E_j&=&T_i^l\nabla_{e_l}T_j^me_m=T_i^lT_j^m\nabla_{e_l}e_m+E_i(T_j^m)e_m \\
&=&T_i^lT_j^m\Gamma_{ml}^ke_k+u_{x_ix_j}e_3,
\end{eqnarray*}
where $\Gamma_{ml}^k$ are the Christoffel symbols of $g$. Therefore,
\[g(N,\nabla_{E_i}E_j)=\,-\,\frac{\sum_{k=1}^2T_i^lT_j^m\Gamma_{ml}^ku_{x_k}}{\sqrt{1+g^{ij}u_{x_i}u_{x_j}}}+\frac{\big(u_{x_ix_j}+T_i^lT_j^m\Gamma_{ml}^3\big)}{\sqrt{1+g^{ij}u_{x_i}u_{x_j}}}
.\]
Finally, the mean curvature equation is written as: 
\begin{equation}\label{mean curvature equation}
H\,\sqrt{1+g^{ij}u_{x_i}u_{x_j}}=h^{ij}u_{x_ix_j}+h^{ij}T_i^lT_j^m\Gamma_{ml}^3-\sum_{k=1}^2h^{ij}T_i^lT_j^mu_{x_k}\Gamma_{ml}^k.
\end{equation}
Since $h_{ij}=T_i^lT_j^mg_{ml}$ and $T_i^3=u_{x_i}$, then $|h^{ij}|\leq C_1(g^{ij},u_{x_i},u_{x_j})$. 
The equation (\ref{mean curvature equation})  is uniformly elliptic as long as $|\nabla u|, |\nabla^2u|< \widetilde{C}$.

Let $p\in M$ be an accumulation point for the sequence $\{\Sigma_n\}$. By the upper bound on the second fundamental form, $\sup_{\Sigma_n}|A_n|<C$,  there exists $r_0=r_0(C)$ such that for every $q \in B_{r_0}(p)\cap \Sigma_n$ we have that $\Sigma_n\cap B_{r_0}(q)$ is locally a graph $u_n$ over a 
neighbourhood $U_{q}\subset T_{q}\Sigma_n$. Moreover, there exists $C_2=C_2(C)>0$ for which $\max\{\nabla u_n, \nabla^2u_n\}\leq C_2$.  As
$|g_n-\delta|_{C^{2,\alpha}}\rightarrow 0$,  the Schauder estimates for  solutions of  elliptic equations, see \cite{GT}, imply that
 $u_n$ have $C^{2,\alpha}$ estimates on $B_{r_0/2}(q)$, i.e., $|u_n|_{C^{2,\alpha}}\leq C_3(|H_n|+ |u_n|)$. Therefore, $u_n,\nabla u_n,\nabla^2u_n$ are uniformly bounded and equicontinuous.

As $\{\Sigma_n\}_{n\in\mathbb{N}}$  satisfy local area bounds, then $|\Sigma_n\cap B(r_0/2)(y)|\leq C_4$.
On the other hand,  the monotonicity formula,  Proposition \ref{monotonicity}, gives that  $|\Sigma_n^j\cap B_{r_0}(y)|\geq C_5 r_0^2$, where  $\Sigma_n^j$ is a connected component of $\Sigma_n\cap B_{r}$.
It follows that the number of components of $\Sigma_n \cap B_{r_0/2}(p)$ is finite and independent of $n$.
By the Arzel\`{a}-Ascoli theorem we can extract a  subsequence for which $\{u_n^j\}$  converges to  $u$  for every $j$. Moreover, $u$ also satisfies  the constant mean curvature equation (\ref{mean curvature equation}). As the set of accumulation points of $\{\Sigma_n\}$ is compact in $B_R(p)$ we can cover this set by finite balls $B_{r_0}(p_k)$ with $k=1,\ldots, N$. Repeating the arguments in each of these balls  and applying a diagonal argument   we obtain a properly immersed surface $\Sigma$ on  $B_R(p)\subset M$ with constant mean curvature $H$. Since
 the surfaces $\Sigma_n$ are embedded, it follows that $\Sigma$ does not cross itself though it may have tangential self-intersections. Therefore, $\Sigma$ is  properly weakly embedded in $M$. 
\end{proof}
\begin{corollary}\label{compactness iso}
\textit{Let $(\Sigma_n, x_n) \subset (M_n,g_n, x_n)$ be a sequence of isoperimetric surfaces with $|A_n|\leq C$. Assume that  $(M_n,g_n,x_n)$ converges,  in the sense of Cheeger-Gromov, to  a three manifold $(M,g,x)$.  There exists a properly embedded surface $\Sigma\subset (M,g,x)$ such that  $\Sigma_n\rightarrow \Sigma$ in the sense of graphs and the convergence is with multiplicity one.}
\end{corollary}
\begin{proof}[Sketch of  the Proof.]
First we remark that $\{\Sigma_n\}$ satisfy local area bounds. Indeed, $\text{Area}(\Sigma_n\cap B_r(p))\leq 2 |\partial B_r(p)|$. 
By Proposition \ref{compactness} we only need  to rule out possibly multiplicities for the convergence $\Sigma_n\rightarrow \Sigma$ and   points where $\Sigma$  fails to be embedded. 
 
If the multiplicity of the limit is bigger than two, then $\Sigma_n \cap B_{r}(p)$ has  several  components getting arbitrarily  close.
This allow us to do  a local cut and past deformation, as shown in Figure 2, that preserves the enclosed volume.    
If $\delta$ is the Euclidean metric, then  $\frac{1}{\widetilde{C}} \delta \leq g_n \leq \widetilde{C} \delta$ and $\frac{1}{C^{\prime}}\text{Area}_{\delta} \leq \text{Area}_{g_n} \leq C^{\prime} \text{Area}_{\delta}$. 
Thus, if $h \ll r$, then \[Area_{g_n}(\Sigma_n^{\prime})\leq Area_{g_n}(\Sigma_n) -C_1^{\prime} r^2 + C_2^{\prime} rh< \text{Area}_{g_n}(\Sigma_n).\] 
\begin{figure}[h]
 \begin{center}
   \includegraphics[scale=0.45]{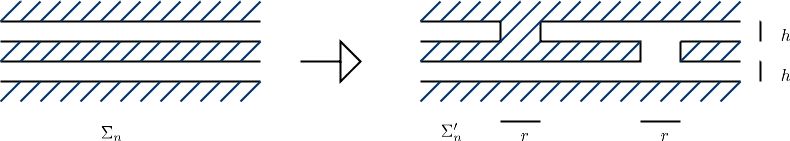}\\
   \caption{Example of higher multiplicity}
   \end{center}
 \end{figure}
 
The deformation needed for the multiplicity two case is shown in Figure 3. 
The constraint on the enclosed volume implies that $\frac{4}{3}\pi R^3\approx \pi r^2h$. Hence, if $h\ll r$, then
\[Area(\Sigma_n^{\prime})\leq\,Area(\Sigma_n^{\prime})- C_1^{\prime} r^2+C_2^{\prime} r\,h-C_3^{\prime} R^2+ C_4^{\prime} R^2 < Area(\Sigma_n^{\prime}).\]

 The construction to deal with points where   $\Sigma$ has tangential self intersections is  similar to the multiplicity two case.  The corollary now follows since these constructions contradict the fact that $\Sigma_n$ is  isoperimetric for every $n$.
 \end{proof}
  \begin{figure}[h]
  \begin{center}
    \includegraphics[scale=0.5]{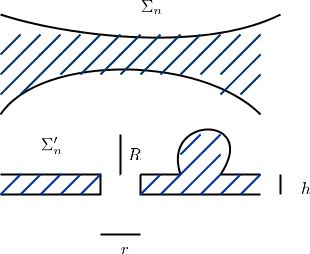}\\
    \caption{Example of multiplicity two}
    \end{center}
  \end{figure}

\begin{theorem}[Korevaar-Kusner-Solomon]\label{kks}
\textit{If $\Sigma$ is a complete CMC surface  properly embedded in a solid cylinder in $\mathbb{R}^3$, then $\Sigma$ is rotationally symmetric with
respect to a line parallel to the axis of the cylinder.}
\end{theorem}
\begin{proof}
See \cite{KKS}.
\end{proof}

\begin{theorem}[Da Silveira]\label{silveira}
\textit{Let $(\Sigma,ds^2)$ be a complete orientable  surface conformally equivalent to a compact Riemann surface punctured at a finite number of points.
Let $L=\Delta + q$ be an operator on $\Sigma$ satisfying $q\geq 0$ and $q \neq 0$. If $\Sigma$ has infinite area, then there exists a piecewise smooth function $f$ with compact support such that
\[-\int_{\Sigma}f\,L\,f \,d_{\Sigma}<0\quad \textit{and}\quad \int_{\Sigma}f\,d_{\Sigma}=0.\]}
\end{theorem}
\begin{proof}
See \cite{DS}.
\end{proof}

\end{section}

\end{document}